\documentclass[12pt]{article}
\usepackage{layout}
\usepackage{url}
\usepackage{proof}
\usepackage{mathrsfs}
\usepackage{amsmath,amsthm,amssymb}
\usepackage{cleveref}
\usepackage[dvipdfmx]{graphicx}
\usepackage[all]{xy}
\usepackage{color}
\usepackage{indentfirst}

\numberwithin{equation}{section}

\newtheorem{thm}{Theorem}[section]
\newtheorem{prop}[thm]{Proposition}
\newtheorem{lemm}[thm]{Lemma}
\newtheorem{cor}[thm]{Corollary}
\newtheorem{defi}[thm]{Definition}

\def\cal#1{\mathcal{#1}}
\def\scr#1{\mathscr{#1}}
\def\rm#1{\mathrm{#1}}
\def\bb#1{\mathbb{#1}}

\def\wh#1{\widehat{#1}}

\def\ol#1{\overline{#1}}

\newcommand{\ds}{\displaystyle}
\newcommand{\QQ}{\mathbb{Q}}
\newcommand{\Qlb}{\overline{\mathbb{Q}}_\ell}

\newcommand{\ZZ}{\mathbb{Z}}

\newcommand{\Hom}{\mathrm{Hom}}

\newcommand{\Rep}{\mathrm{Rep}}
\newcommand{\Vect}{\mathrm{Vect}}
\newcommand{\LocSys}{\mathrm{LocSys}}

\newcommand{\Sat}{\mathrm{Sat}}
\newcommand{\FS}{\mathrm{FS}}
\newcommand{\YZ}{\mathrm{YZ}}
\newcommand{\Hck}{\mathcal{H}\mathrm{ck}}
\newcommand{\HHom}{\mathscr{H}\mathrm{om}}
\newcommand{\Div}{\mathrm{Div}}
\newcommand{\ULA}{\mathrm{ULA}}

\newcommand{\et}{\mathrm{\acute{e}t}}
\newcommand{\Det}{D_{\mathrm{\acute{e}t}}}
\newcommand{\bd}{\mathrm{bd}}
\newcommand{\CT}{\mathrm{CT}}
\newcommand{\dimtrg}{\mathrm{dim.trg}}

\newcommand{\Gr}{\mathrm{Gr}}

\DeclareMathOperator{\Spec}{Spec}
\DeclareMathOperator{\Spd}{Spd}

\title{\Large{Relation between the two geometric Satake equivalences via nearby cycle}}
\author{Katsuyuki Bando\footnote{Katsuyuki Bando, Research Institute for Mathematical Sciences, Kyoto University, Kyoto, Japan, \texttt{kbando@kurims.kyoto-u.ac.jp}}}

\begin{document}
\maketitle
\renewcommand{\thefootnote}{\fnsymbol{footnote}}
\footnote[0]{Proposed running head: Relation between two geometric Satake via nearby cycle}
\renewcommand{\thefootnote}{\arabic{footnote}}
\abstract{Fargues and Scholze proved the geometric Satake equivalence over the Fargues--Fontaine curve.
On the other hand, Zhu proved the geometric Satake equivalence using a Witt vector affine Grassmannian.
In this paper, we explain the relation between the two versions of the geometric Satake equivalence via nearby cycle.}
\section{Introduction}
Let $F$ be a $p$-adic local field with a residue field $\bb{F}_q$.
Write $\cal{O}_F$ for its ring of integers.
Put $k:=\ol{\bb{F}}_q$, an algebraic closure of $\bb{F}_q$.
Let $G$ be a reductive group scheme over $\cal{O}_F$.
Let $\Lambda$ be either $\Qlb$, a finite extension $L$ of $\QQ_{\ell}$, its ring of integers $\cal{O}_L$, or its quotient ring.
By \cite[Theorem I.6.3]{FS}, we have an equivalence of symmetric monoidal categories
\[
\scr{S}_{\FS}\colon \Sat(\Hck_{G,\Div^1}, \Lambda)\overset{\sim}{\longrightarrow} \Rep(\wh{G}^{\rm{tw}}\rtimes W_F,\Lambda)
\]
where $\Sat(\Hck_{G,\Div^1}, \Lambda)$ is the Satake category, which is the monoidal category of perverse sheaves with coefficients in $\Lambda$ on the Hecke stack $\Hck_{G, \Div^1}$ over the moduli space $\Div^1=\Spd \breve{F}/\varphi^{\mathbb{Z}}$ parametrizing degree 1 Cartier divisors on the Fargues--Fontaine curve for $F$.
Let $\wh{G}$ be the Langlands dual group over $\Lambda$ and $W_F$ the Weil group of $F$.
The group scheme $\wh{G}^{\rm{tw}}\rtimes W_F$ is a semidirect product of $\wh{G}$ and $W_F$ with respect to the $W_F$-action on $\wh{G}$ obtained by twisting the usual action by an explicit cyclotomic action.
Then $\Rep(\wh{G}^{\rm{tw}}\rtimes W_F,\Lambda)$ is the monoidal category of $\wh{G}^{\rm{tw}}\rtimes W_F$-representations on finite free $\Lambda$-modules.

Consider a diagram
\[
\xymatrix{
\Hck_{G,\Spd k}\ar[r]^-{i}&\Hck_{G,\Spd \cal{O}_{\wh{\ol{F}}}}&\Hck_{G,\Spd \wh{\ol{F}}}\ar[l]_-{j}\ar[d]^{p_{F}}\\
&&\Hck_{G,\Div^1}, 
}
\]
see \S \ref{scn:notation} for the notations.
We have a nearby cycle functor
\[
\Psi:=i^* \circ Rj_* \circ (p_{F})^*\colon \Sat(\Hck_{G,\Div^1},\Lambda)\to \Sat(\Hck_{G,\Spd k},\Lambda).
\]
The main theorem is the following:
\begin{thm}
There exist a symmetric monoidal strcuture on $\Sat(\Hck_{G,\Spd k},\Lambda)$ whose underlying monoidal structure is given by the convolution product $-\star -$, and a symmetric monoidal equivalence
\[
\scr{S}_{\YZ}\colon \Sat(\Hck_{G,\Spd k},\Lambda)\overset{\sim}{\longrightarrow} \Rep(\wh{G},\Lambda)
\]
such that for the diagram of functors
\[
\xymatrix{
\Sat(\Hck_{G,\Div^1}, \Lambda)\ar[r]^-{\scr{S}_{\FS}}\ar[d]_{\Psi}&\Rep(\wh{G}^{\rm{tw}}\rtimes W_F,\Lambda)\ar[d]^{\rm{For}}\\
\Sat(\Hck_{G,\Spd k},\Lambda)\ar[r]_-{\scr{S}_{\YZ}}&\Rep(\wh{G},\Lambda),
}
\]
there is a natural isomorphism
\[
\rm{For}\circ \scr{S}_{\FS}\cong \scr{S}_{\YZ}\circ \Psi.
\]
\end{thm}
The equivalence $\scr{S}_{\YZ}$ with coefficients in $\Qlb$, $\bb{F}_{\ell}$ or $\ZZ_{\ell}$ has the same form as the geometric Satake equivalence in \cite{Zhumixed} or \cite{Yu}.
We do not prove that the equivalence $\scr{S}_{\YZ}$ is the same functor as the equivalence in \cite{Zhumixed} or \cite{Yu}, which is too complicated to prove since the way of constructing a symmetric monoidal structure on Satake categories is quite different between \cite{FS} and \cite{Zhumixed}.
However, we prove that the several properties of the equivalence given in \cite{Zhumixed} also hold for $\scr{S}_{\YZ}$.

The similar results to the above theorem are mentioned in \cite[Remark I.2.14]{FS}.
However, we show further that the nearby cycle functor, involving $Rj_*$ can be used as the vertical functor, while \cite{FS} used $(j^*)^{-1}$ instead.

While we are completing this work, a related paper \cite{AGLR} appeared.
The results of Proposition \ref{prop:perfconstconti}, Theorem \ref{thm:pullackequiv} and Corollary \ref{cor:preserveSat} overlap with \cite[Lemma 6.11, Proposition 6.12, Corollary 6.14]{AGLR}, respectively.
However, our proofs of these results are different from that of \cite{AGLR}.
We give a deeper discussion on the fundamental theory of sheaves on $\Spd k, \Spd \wh{\ol{F}}$, and $\Spd \cal{O}_{\wh{\ol{F}}}$ in \S \ref{ssc:genonetcoh}.
Also, our proofs are simpler in some way than \cite{AGLR} since the result of \cite{AGLR} is more general.
\subsection*{Acknowledgement}
The author is really grateful to his advisor Naoki Imai for useful discussions and comments.
He also thanks to referees for careful reading of the paper and valuable suggestions and comments.
This work is supported by JSPS KAKENHI Grant Number 2KJ1105.
\section{Notations}\label{scn:notation}
For a field $K$, let $\overline{K}$ denote an algebraic closure of $K$.

Throughout this paper, $F$ is a $p$-adic local field with a residue field $\bb{F}_q$.
Write $\cal{O}_F$ for its ring of integers.
Put $k:=\ol{\bb{F}}_q$.

Let $\ell$ be a prime number not equal to $p$.
Let $\Lambda$ be either $\Qlb$, a finite extension $L$ of $\QQ_{\ell}$, its ring of integers $\cal{O}_L$, or its quotient ring.
The category of finite free $\Lambda$-module is denoted by $\rm{Vect}(\Lambda)$.

Put 
\[
\Div^1:=\Spd \breve{F}/\varphi^{\bb{Z}}
\]
where $\breve{F}$ is a maximal unramified extension of $F$, and $\varphi$ is the Frobenius.
Let $G$ be a reductive group over $\Spec \cal{O}_F$.
For a small v-stack $S$ with a map $S\to \Div^1$, or $S\to \Spd \cal{O}_{\breve{F}}$, let $\Gr_{G,S}$ be the Beilinson--Drinfeld affine Grassmannian of $G$ over $S$ and $\Hck_{G,S}$ the local Hecke stack defined in \cite[VI.1]{FS}.

The Langlands dual group of $G$ over $\Lambda$ is denoted by $\wh{G}_{\Lambda}$ or simply $\wh{G}$.
The scheme $\wh{G}$ has a natural $W_F$-action.
We define the category $\Rep(\wh{G},\Lambda)$ as the monoidal category of algebraic $\wh{G}$-representations on finite free $\Lambda$-modules.
We write $\Rep(\wh{G}\rtimes W_F,\Lambda)$ for the category of $\wh{G}\rtimes W_F$-representations on finite free $\Lambda$-modules such that the $\wh{G}$-action is algebraic, and that the $W_F$-action is continuous with respect to the $\ell$-adic topology on $\Lambda$.
The category $\Rep(\wh{G}^{\rm{tw}}\rtimes W_F,\Lambda)$ is similarly defined with respect to the twisted $W_F$-action on $\wh{G}$ appearing in \cite[Theorem I.6.3]{FS}.

If $\Lambda$ is a torsion ring, then the derived category $\Det(X,\Lambda)$ for a small v-stack $X$ and the six functors with respect to this formalism are defined as in \cite[Definition 1.7]{Sch}.
\section{Torsion coefficient case}
Let $\Lambda$ be a proper quotient of $\cal{O}_L$ where $L$ is a finite extension over $\QQ_{\ell}$.
\subsection{Generalities on \'{e}tale cohomology of diamonds}\label{ssc:genonetcoh}
Put $C:=\wh{\ol{F}}$.
\subsubsection{Excision distinguished triangle}
\begin{lemm}\label{lemm:RHomdist}
Let $X$ be a small v-stack.
For $A\in \Det(X,\Lambda)$, the functor
\[
R\HHom_{\Det(X,\Lambda)}(-,A)\colon \Det(X,\Lambda)^{\rm{op}}\to \Det(X,\Lambda)
\]
preserves distinguished triangles.
\end{lemm}
\begin{proof}
Following \cite[Corollary 17.2]{Sch}, write
\[
R_{X_{\et}}\colon D(X_v,\Lambda)\to \Det(X,\Lambda)
\]
for the right adjoint of the inclusion $\Det(X,\Lambda)\hookrightarrow D(X_v,\Lambda)$.
Then as in the remark just after Lemma 17.8 of \cite{Sch}, we have a natural isomorphism
\[
R\HHom_{\Det(X,\Lambda)}(-,A)\cong R_{X_{\et}}R\HHom_{D(X_v,\Lambda)}(-,A)\colon \Det(X,\Lambda)^{\rm{op}}\to \Det(X,\Lambda).
\]
The functor $R\HHom_{D(X_v,\Lambda)}(-,A)$ preserves distinguished triangles since it is the derived functor of the inner hom of v-sheaves.
The functor $R_{X_{\et}}$ preserves distinguished triangles since its left adjoint $\Det(X,\Lambda)\hookrightarrow D(X_v,\Lambda)$ preserves distinguished triangles.
\end{proof}
Let $j\colon \Spd C\to \Spd \cal{O}_C$ and $i\colon \Spd k\to \Spd \cal{O}_C$ be inclusions.
\begin{lemm}
For $A\in \Det(\Spd \cal{O}_C,\Lambda)$,
\begin{align}
&j_!j^*A\to A\to i_*i^*A\to, \label{eqn:disttri}\\
\text{and\quad }&i_*Ri^!A\to A\to Rj_*j^*A\to \label{eqn:disttri2}
\end{align}
are distinguished triangles.
\end{lemm}
\begin{proof}
The functors $j^*, i^*$ are exact (that is, commute with canonical truncations) by definition.
The functor $Rj_!$ is also exact (see \cite[Definition/Proposition 19.1]{Sch}), so $Rj_!=j_!$.
Moreover, the functor $Ri_*$ is exact (by the fact that $\dimtrg i=0$ and \cite[Theorem 22.5]{Sch}, for example), so $Ri_*=i_*$.
Hence it suffices to show that for $A\in \Det(\Spd \cal{O}_C,\Lambda)$ concentrated in degree 0, 
\[
0\to j_!j^*A\to A\to i_*i^*A\to 0
\]
is an exact sequence.
This can be checked after taking $i^*$ and $j^*$, and that is easy.

For (\ref{eqn:disttri2}), by (\ref{eqn:disttri}) and Lemma \ref{lemm:RHomdist}, we have a distinguished triangle
\[
R\HHom(i_*i^*\Lambda, A)\to R\HHom(\Lambda, A)\to R\HHom(Rj_!j^*\Lambda, A)\to.
\]
By \cite[Theorem 1.8(iv), (v)]{Sch}, we have
\begin{align*}
R\HHom(i_*i^*\Lambda, A)&\cong i_*R\HHom(i^*\Lambda, Ri^!A)\cong i_*Ri^!R\HHom(\Lambda,A)\cong i_*Ri^!A\\
R\HHom(\Lambda, A)\cong A\\
R\HHom(j_!j^*\Lambda, A)&\cong Rj_*R\HHom(j^*\Lambda, Rj^!A)\cong Rj_*Rj^!R\HHom(\Lambda,A)\cong Rj_*j^*A,
\end{align*}
hence the lemma follows.
\end{proof}
Let $D(\Lambda\mathchar`-\mathrm{mod})$ be the derived category of $\Lambda$-modules.
\begin{lemm}\label{lemm:Ckconst}
The natural functors
\begin{align*}
D(\Lambda\mathchar`-\mathrm{mod})&\to \Det(\Spd C,\Lambda),\\
D(\Lambda\mathchar`-\mathrm{mod})&\to \Det(\Spd k,\Lambda)
\end{align*}
which maps a complex $K$ of $\Lambda$-module to the constant sheaf associated with $K$, are equivalences.
\end{lemm}
\begin{proof}
The first equivalence follows from the equivalence
\[
\Det(\Spd C,\Lambda)\cong D(|\Spd C|,\Lambda)=D(\Lambda\mathchar`-\mathrm{mod})
\]
since $\Spd C$ is a strictly totally disconnected perfectoid space.
For the second one, the composition
\[
D(\Lambda\mathchar`-\mathrm{mod})\to \Det(\Spd k,\Lambda) \to \Det(\Spd C,\Lambda)
\]
is equivalence by the first one, where the second map is the pullback along the projection $\Spd C\to \Spd k$.
The claim follows from the fact that the second map is fully faithful by \cite[Propositon 19.5(ii)]{Sch}.
\end{proof}
\begin{lemm}\label{lemm:cohon3}
For $K\in D(\Lambda\mathchar`-\mathrm{mod})$, the natural maps
\begin{align*}
K&\to R\Gamma(\Spd k,K),\\
K&\to R\Gamma(\Spd \cal{O}_C,K),\\
K&\to R\Gamma(\Spd C,K)
\end{align*}
are isomorphism in $D(\Lambda\mathchar`-\mathrm{mod})$.
Here we still write $K$ for the constant sheaf associated with $K$.
\end{lemm}
\begin{proof}
The first and third claims follow from Lemma \ref{lemm:Ckconst}.
The second follows from the isomorphism
\[
R\Gamma(\Spd \cal{O}_C,K)\cong R\Gamma(\Spd k,K)
\]
in \cite[Remark V.4.3 (ii)]{FS}.
\end{proof}
\begin{prop}\label{prop:perfconstconti}
For $K\in D(\Lambda\mathchar`-\mathrm{mod})$, the natural map
\[
K\to Rj_*K \in \Det(\Spd \cal{O}_C,\Lambda)
\]
is isomorphism.
Here we still write $K$ for the constant sheaf associated with $K$.
\end{prop}
\begin{proof}
The distinguished triangle
\[
R\Gamma i_*Ri^! K\to R\Gamma K \to R\Gamma Rj_*j^*K\to
\]
induced by (\ref{eqn:disttri2}) is isomorphic to
\[
Ri^! K \to K \overset{\mathrm{id}}{\to} K \to
\]
by Lemma \ref{lemm:cohon3}.
Thus we get $Ri^!K=0$.
Considering the distinguished triangle (\ref{eqn:disttri2}) again, it follows that
\[
K\to Rj_*K
\]
is an isomorphism.
\end{proof}
\begin{cor}\label{cor:equivlc}
There is an equivalence
\[
D_{\rm{lc}}(\Spd k,\Lambda)\underset{i^*}{\overset{\sim}{\leftarrow}} D_{\rm{lc}}(\Spd \cal{O}_C,\Lambda)\underset{j^*}{\overset{\sim}{\to}}D_{\rm{lc}}(\Spd C,\Lambda)\simeq D_{\rm{perf}}(\Lambda\mathchar`-\mathrm{mod}),
\]
where $D_{\rm{lc}}(-,\Lambda)\subset \Det(-,\Lambda)$ denotes the full subcategory locally constant sheaves with perfect fibers, and $D_{\rm{perf}}(\Lambda\mathchar`-\mathrm{mod})$ is the category of perfect $\Lambda$-complexes.

Moreover, the quasi-inverse of the above functor $j^*$ is given by $Rj_*$
\end{cor}
\begin{proof}
The first statement is mentioned in the proof of \cite[Corollary VI.6.7]{FS}.
However, we give a proof:\\
The equivalence $D_{\rm{lc}}(\Spd C,\Lambda)\simeq D_{\rm{perf}}(\Lambda\mathchar`-\mathrm{mod})$ follows from Lemma \ref{lemm:Ckconst}.
By the definition of locally constant sheaves, objects of the three categories of sheaves are constant sheaves.
Thus it suffices to show that $i^*$ and $j^*$ are fully faithful.
Since $\Hom(K_1,K_2)=\Hom(\Lambda,\HHom(K_1,K_2))$ holds and $\HHom(K_1,K_2)$ is constant if $K_1,K_2$ are constant, this follows from Lemma \ref{lemm:cohon3}.

The second statement follows from Proposition \ref{prop:perfconstconti}.
\end{proof}
\subsection{Nearby cycle functor}\label{ssc:nearby}
For a small v-stack $S$ over $\Div^1$ or $\Spd \cal{O}_{\breve{F}}$, we refer to \cite[Definition VI.6.1]{FS} for the notion of universally locally acyclic (ULA) sheaves on $\Hck_{G,S}$ over a torsion coefficient ring $\Lambda$.
\begin{prop} \label{prop:ULAprsv}
If $K\in \Det(\Hck_{G,\Spd C},\Lambda)$ is ULA over $\Spd C$, then $Rj_*K\in \Det(\Hck_{G,\Spd \cal{O}_C},\Lambda)$ is ULA over $\Spd \cal{O}_C$.
\end{prop}
\begin{proof}
Since we are working over $\Spd \cal{O}_C$, we may assume that $G$ is split.
Fix a maximal split torus $T\subset G$ and a Borel subgroup $B\subset G$ containing $T$.
By \cite[Proposition VI.6.4]{FS}, for a small v-stack $S$ over $\Div^1$ or $\Spd \cal{O}_{\breve{F}}$ and $A\in \Det(\Hck_{G,S},\Lambda)^{\bd}$, the complex $A$ is ULA over $S$ if and only if
\[
R\pi_{T,S*}\CT_{B}(A)\in \Det(S,\Lambda)^{\bd}
\]
is locally constant with perfect fibers, where $\pi_{T,S}\colon \Gr_{T,S}\to S$ is the projection and $\CT_B$ is the hyperbolic localization functor.
Since $\CT_B$ can be described as a composition of $(-)_*$ and $(-)^!$ by \cite[Theorem VI.6.5]{FS}, $\CT_B$ is compatible with $Rj_*$.

Thus it suffices to show that if $K\in \Det(\Spd C,\Lambda)$ is locally constant with perfect fibers, then $Rj_*K \in \Det(\Spd \cal{O}_C,\Lambda)$ is locally constant with perfect fibers.
This follows from Corollary \ref{cor:equivlc}.
\end{proof}
\begin{thm} \label{thm:pullackequiv}
The pullback functor
\[
j^*\colon \Det^{\ULA}(\Hck_{G,\Spd \cal{O}_C},\Lambda)^{\bd}\to \Det^{\ULA}(\Hck_{G,\Spd C},\Lambda)^{\bd}
\]
is an equivalence of categories, and its quasi-inverse is $Rj_*$.
\end{thm}
\begin{proof}
The fact that $j^*$ is an equivalence is proved in \cite[Corollary VI.6.7]{FS}.
By Proposition \ref{prop:ULAprsv} and the identity $j^*Rj_*=\mathrm{id}$, its quasi-inverse is $Rj_*$.
\end{proof}
\begin{cor}\label{cor:preserveSat}
Consider the diagram
\[
\xymatrix{
\Hck_{G,\Spd k}\ar[r]^-{i}&\Hck_{G,\Spd \cal{O}_{\wh{\ol{F}}}}&\Hck_{G,\Spd \wh{\ol{F}}}\ar[l]_-{j}\ar[d]^{p_{F}}\\
&&\Hck_{G,\Div^1}.
}
\]
The functors $(p_F)^*$, $Rj_*$, $i^*$ preserves the objects of the Satake categories.
Moreover, $Rj_*$ and $i^*$ induce equivalences
\begin{align*}
Rj_*&\colon \Sat(\Hck_{G, \Spd C},\Lambda) \overset{\sim}{\to} \Sat(\Hck_{G, \Spd \cal{O}_C},\Lambda),\\
i^*&\colon \Sat(\Hck_{G, \Spd \cal{O}_C},\Lambda) \overset{\sim}{\to} \Sat(\Hck_{G, \Spd k},\Lambda).
\end{align*}
\end{cor}
\begin{proof}
Since we are considering the relative local acyclicity and the relative perversity, $i^*, j^*, (p_F)^*$ preserves ULA objects and flat perverse (i.e. for all $\Lambda$-modules $M$, $A\otimes^{\bb{L}}_{\Lambda}M$ is perverse) objects.

Let us prove that $Rj_*$ preserves the objects of the Satake categories.
Since we are working over $\Spd \cal{O}_C$, we may assume that $G$ is split.
Fix a maximal split torus $T\subset G$ and a Borel subgroup $B\subset G$ containing $T$.
The pushforward $Rj_*$ commutes with $R\pi_{T,S*}\CT_B(-)[\deg]$ where $\pi_{T,S}\colon \Gr_{T,S}\to S$ is the projection.
Hence by \cite[Proposition VI.6.4, Proposition VI.7.7]{FS}, we need to show that if $A\in \Det(\Spd C,\Lambda)$ is \'{e}tale locally a finite free $\Lambda$-module in degree 0, then so is $Rj_*A\in \Det(\Spd\cal{O}_C,\Lambda)$.
This follows from Proposition \ref{prop:perfconstconti} as in the proof of Proposition \ref{prop:ULAprsv}.

Moreover, $Rj_*$ induces an equivalence since $j^*$ gives the quasi-inverse.
It remains to show that $i^*$ induces an equivalence.
As in \cite[Corollary VI.6.7]{FS}, 
\[
i^*\colon \Det^{\ULA}(\Hck_{G, \Spd \cal{O}_C},\Lambda) \overset{\sim}{\to} \Det^{\ULA}(\Hck_{G, \Spd k},\Lambda)
\]
is an equivalence.
By \cite[Proposition VI.7.7]{FS}, an object $A\in \Det^{\ULA}(\Hck_{S},\Lambda)$ is an object of $\Sat(\Hck_{S},\Lambda)$ if and only if $R\pi_{T,S*}\CT_B(A)[\deg]$ is locally constant with finite free fibers concentrated in degree 0.
Thus, it suffices to show that for $A\in D_{\rm{lc}}(\Spd\cal{O}_C,\Lambda)$, the object $i^*A\in \Det(\Spd k,\Lambda)$ is locally constant with finite free fibers concentrated in degree 0 if and only if so is $A$.
This follows from Corollary \ref{cor:equivlc}
\end{proof}
Therefore, we have a functor
\begin{equation}\label{defofPsitor}
\Psi:=i^* Rj_* (p_F)^*\colon \Sat(\Hck_{G, \Div^1},\Lambda)\to \Sat(\Hck_{G, \Spd k},\Lambda)
\end{equation}
called a nearby cycle functor.
\subsection{Nearby cycle functor and convolution}\label{ssc:monoidal}
For a v-stack $S$ over $\Div^1$ or $\Spd \cal{O}_{\breve{F}}$, there is a monoidal structure on
\[
\Sat(\Hck_{G,S},\Lambda)
\]
given by the convolution product $-\star -$ defined in \cite[VI.8]{FS}.
In this subsection, we prove that
\begin{prop}\label{prop:monoidal}
Consider the nearby cycle functor
\[
\Psi:=i^* Rj_* (p_F)^*\colon \Sat(\Hck_{G,\Div^1},\Lambda)\to \Sat(\Hck_{G,\Spd k},\Lambda).
\]
There is a natural isomorphism
\[
\Psi(A\star B)\cong \Psi(A)\star \Psi(B)
\]
for $A,B\in \Sat(\Hck_{G,\Div^1},\Lambda)$.
\end{prop}
Applying Theorem \ref{thm:pullackequiv} to a reductive group $G\times_{\cal{O}_F}G$, noting that there is a canonical isomorphism
\[
\Hck_{G,S}\times_S\Hck_{G,S}\cong \Hck_{G\times_{\cal{O}_F} G,S},
\]
we have the following proposition:
\begin{prop}\label{prop:pullbackequivGtimesG}
Let $j_{G^2}\colon \Hck_{G,\Spd C}\times_{\Spd C} \Hck_{G,\Spd C} \to \Hck_{G,\Spd \cal{O}_C}\times_{\Spd \cal{O}_C} \Hck_{G,\Spd \cal{O}_C}$ be the natural inclusion.
Then 
\begin{align*}
j_{G^2}^*\colon &\Det^{\ULA}(\Hck_{G,\Spd \cal{O}_C}\times_{\Spd \cal{O}_C} \Hck_{G,\Spd \cal{O}_C},\Lambda)\\
&\to \Det^{\ULA}(\Hck_{G,\Spd C}\times_{\Spd C} \Hck_{G,\Spd C},\Lambda)
\end{align*}
is an equivalence of categories, and its quasi-inverse is $Rj_{G^2*}$.
\end{prop}
\begin{proof}[Proof of Proposition \ref{prop:monoidal}]
For a small v-stack $S\to \Div^1_{\cal{Y}}$, put 
\[
\widetilde{\Hck}_{G,S}:=L^+_{S}G\backslash L_{S}G\times^{L^+_{S}G}L_{S}G/L^+_{S}G,
\]
Consider the diagram
\[
\xymatrix{
\Hck_{G,\Div^1}\times_{\Div^1} \Hck_{G,\Div^1}\ar@{<-}[r]^-{a_{\Div^1}}\ar@{<-}[d]_{p_{F,G^2}}\ar@{}[rd]|{\square}&\widetilde{\Hck}_{G,\Div^1}\ar[r]^-{b_{\Div^1}}\ar@{<-}[d]\ar@{}[rd]|{\square}&\Hck_{G,\Div^1}\ar@{<-}[d]^{p_F}\\
\Hck_{G,\Spd C}\times_{\Spd C} \Hck_{G,\Spd C}\ar@{<-}[r]^-{a_{C}}\ar[d]_{j_{G^2}}\ar@{}[rd]|{\square}&\widetilde{\Hck}_{G,\Spd C}\ar[r]^-{b_{C}}\ar[d]\ar@{}[rd]|{\square}&\Hck_{G,\Spd C}\ar[d]^j\\
\Hck_{G,\Spd \cal{O}_C}\times_{\Spd \cal{O}_C} \Hck_{G,\Spd \cal{O}_C}\ar@{<-}[r]^-{a_{\cal{O}_C}}\ar@{<-}[d]_{i_{G^2}}\ar@{}[rd]|{\square}&\widetilde{\Hck}_{G,\Spd \cal{O}_C}\ar[r]^-{b_{\cal{O}_C}}\ar@{<-}[d]\ar@{}[rd]|{\square}&\Hck_{G,\Spd \cal{O}_C}\ar@{<-}[d]^{i}\\
\Hck_{G,\Spd k}\times_{\Spd k} \Hck_{G,\Spd k}\ar@{<-}[r]^-{a_{k}}&\widetilde{\Hck}_{G,\Spd k}\ar[r]^-{b_{k}}&\Hck_{G,\Spd k}.
}
\]
Here the functor
\[
a_S:\widetilde{\Hck}_{G,S}\to \Hck_{G,S}\times_{S} \Hck_{G,S}
\]
is induced from the natural projection $L_{S}G\times_S L_{S}G\to \Hck_{G,S}\times_{S} \Hck_{G,S}$, and 
\[
b_S:\widetilde{\Hck}_{G,S}\to \Hck_{G,S}
\]
is induced from the multiplication $L_{S}G\times_S L_{S}G\to L_SG$.

Since $a_{\cal{O}_C}$ is an $L^+G$-torsor and $b_{\Div^1}, b_{\cal{O}_C}$ are ind-proper, by using base changes, we have
\[
R(b_k)_*(a_k)^*(i_{G^2})^*R(j_{G^2})_*(p_{F,G^2})^*(A\boxtimes_{\Lambda,\Div^1} B)\cong \Psi(A\star B)
\]
for $A,B\in \Det(\Hck_{G,\Div^1},\Lambda)^{\bd}$.
It suffices to show that if $A, B\in \Det(\Hck_{G,\Spd C},\Lambda)$ are ULA complexes, then
\begin{align}\label{eqn:RjstarKunneth}
R(j_{G^2})_*(A\boxtimes_{\Lambda,\Spd C} B)\cong Rj_*A\boxtimes_{\Lambda,\Spd \cal{O}_C} Rj_*B.
\end{align}
First, if $A, B\in \Det(\Hck_{G,\Spd C},\Lambda)$ are ULA over $\Spd C$, then 
\[
A\boxtimes_{\Lambda,\Spd C} B\in \Det(\Hck_{G,\Spd C}\times_{\Spd C} \Hck_{G,\Spd C},\Lambda)^{\bd}
\]
is ULA over $\Spd C$.
In fact, by \cite[Proposition VI.6.5]{FS}, it suffices to show that if $A_1,A_2\in \Det(\Spd C,\Lambda)$ are locally constant with perfect fibers, then so is $A_1\boxtimes_{\Lambda,\Spd C} A_2=A_1\otimes_{\Lambda}^{\bb{L}} A_2\in \Det(\Spd C,\Lambda)$.
This is clear as $\Det(\Spd C,\Lambda)\simeq D(\Lambda\mathchar`-\rm{mod})$.

Therefore, by Proposition \ref{prop:pullbackequivGtimesG}, we only have to show (\ref{eqn:RjstarKunneth}) after taking $j_{G_2}^*$.
Since pullbacks are compatible with exterior tensor products, we have
\begin{align*}
j_{G_2}^*R(j_{G^2})_*(A\boxtimes_{\Lambda,\Spd C} B)&\cong A\boxtimes_{\Lambda,\Spd C} B\\
&\cong j^*Rj_*A\boxtimes_{\Lambda,\Spd C} j^*Rj_*B\\
&\cong j_{G_2}^*(Rj_*A\boxtimes_{\Lambda,\Spd \cal{O}_C} Rj_*B),
\end{align*}
and the proposition follows.
\end{proof}
\subsection{Relation between the two geometric Satake via nearby cycle} \label{ssc:TorsionRelation}
\begin{thm}\label{thm:F}
Consider the diagram
\begin{align}
\xymatrix{
\Sat(\Hck_{G,\Div^1},\Lambda)\ar[r]^-{F_{\Div^1}}\ar[d]_{(p_F)^*}&\Rep(W_F,\Lambda)\ar[d]^{\rm{For}} \\
\Sat(\Hck_{G,\Spd C},\Lambda)\ar[r]^-{F_{\Spd C}}\ar[d]_{Rj_*}&\rm{Vect}(\Lambda)\ar@{=}[d] \\
\Sat(\Hck_{G,\Spd \cal{O}_C},\Lambda)\ar[r]^-{F_{\Spd \cal{O}_C}}\ar[d]_{i^*}&\rm{Vect}(\Lambda)\ar@{=}[d]\\
\Sat(\Hck_{G,\Spd k},\Lambda)\ar[r]_-{F_{\Spd k}}&\rm{Vect}(\Lambda),
}
\end{align}
where
\begin{align*}
F_S:=F_{G,S}\colon \Sat(\Hck_{G,S},\Lambda)&\to \LocSys(S,\Lambda),\\
A&\mapsto \bigoplus_{i} \cal{H}^i(R\pi_{G,S*}A)
\end{align*}
for a small v-stack $S$ over $\Div^1$ or $\Spd \cal{O}_{\breve{F}}$.
Here $\LocSys(S,\Lambda)$ is the category of $\Lambda$-local systems on $S$, and we use the equivalences
\[
\LocSys(\Div^1)\simeq \Rep(W_F,\Lambda), \LocSys(\Spd C)\simeq \Vect(\Lambda),
\]
and similar results. (This follows from \cite[VI.9.2]{FS} and the equivalence in Corollary \ref{cor:equivlc}).

Then there are isomorphisms $\rm{For}\circ F_{\Div^1}\cong F_{\Spd C}\circ (p_F)^*$ and $F_{\Spd C}\cong F_{\Spd k}\circ Rj_*i^*$.
\end{thm}
\begin{proof}
We can prove this by applying the base change results to the cartesian squares
\[
\xymatrix{
\Hck_{G,\Div^1}\ar@{<-}[r]\ar[d]\ar@{}|{\square}[rd]&\Hck_{G,\Spd C}\ar[r]\ar[d]\ar@{}|{\square}[rd]&\Hck_{G, \Spd \cal{O}_C}\ar@{<-}[r]\ar[d]\ar@{}|{\square}[rd]&\Hck_{G, \Spd k}\ar[d]\\
\Div^1\ar@{<-}[r]&\Spd C\ar[r]&\Spd \cal{O}_C\ar@{<-}[r]&\Spd k.
}
\]
\end{proof}
\begin{cor}\label{cor:symmonF}
There exist symmetric monoidal structures on the category 
\[
(\Sat(\Hck_{G, \Spd k},\Lambda),\star)
\]
 and on the functors
\begin{align*}
i^*Rj_*&\colon \Sat(\Hck_{G, \Spd C},\Lambda)\to \Sat(\Hck_{G, \Spd k},\Lambda),\\
F_{\Spd k}&\colon \Sat(\Hck_{G, \Spd k},\Lambda)\to \rm{Vect}(\Lambda)
\end{align*}
such that the isomorphism
\[
F_{\Spd k}\circ i^*Rj_*\cong F_{\Spd C}
\]
in Theorem \ref{thm:F} is monoidal, with respect to the monoidal structure of $F_{\Spd C}$ using fusion products (see \cite[VI.9]{FS}).
\end{cor}
\begin{proof}
The pullback functors $i^*$ and $j^*$ are monoidal, by the standard argument on the convolution product.
Since $i^*Rj_*=i^*(j^*)^{-1}$ holds by Theorem \ref{thm:pullackequiv}, the functor $i^*Rj_*$ is also monoidal.

Thus by pulling back the symmetric monoidal structure on $\Sat(\Hck_{G, \Spd C},\Lambda)$ to $\Sat(\Hck_{G, \Spd k},\Lambda)$ by the monoidal equivalence $i^*Rj_*$, we get the desired structure.
The isomorphism $F_{\Spd k}\circ i^*Rj_*\cong F_{\Spd C}$ is monoidal by Theorem \ref{thm:F}.
\end{proof}
The main theorem follows from this corollary.
\begin{thm}\label{thm:main}
There exist symmetric monoidal structures on the category 
\[
(\Sat(\Hck_{G,\Spd k},\Lambda),\star)
\]
and on the functor
\begin{align*}
F_{\Spd k}&\colon \Sat(\Hck_{G,\Spd k},\Lambda)\to \rm{Vect}(\Lambda)
\end{align*}
which induce a symmetric monoidal equivalence
\[
\scr{S}_{\YZ}\colon \Sat(\Hck_{G,\Spd k},\Lambda)\overset{\sim}{\longrightarrow} \Rep(\wh{G},\Lambda)
\]
such that the squares
\[
\xymatrix{
\Sat(\Hck_{G,\Div^1}, \Lambda)\ar[r]^-{\scr{S}_{\FS}}\ar[d]_{(p_F)^*}&\Rep(\wh{G}^{\rm{tw}}\rtimes W_F,\Lambda)\ar[d]^{\rm{For}}\\
\Sat(\Hck_{G,\Spd C}, \Lambda)\ar[r]^-{\scr{S}_{\FS}}\ar[d]_{i^*Rj_*}&\Rep(\wh{G},\Lambda)\ar@{=}[d]\\
\Sat(\Hck_{G,\Spd k},\Lambda)\ar[r]_-{\scr{S}_{\YZ}}&\Rep(\wh{G},\Lambda)
}
\]
naturally commute.
\end{thm}
\begin{proof}
Consider applying Tannakian reconstructions to each of the horizontal morphisms in Corollary \ref{cor:symmonF}, see \cite[Proposition 11.1]{MV} or \cite[Proposition VI.10.2]{FS}.
Then the theorem follows from the functoriality of the Tannakian reconstruction.
\end{proof}
\subsection{Monoidal structure on hyperbolic localization functor}
\label{ssc:TorsionMonoidal}
Let $P$ be a parabolic subgroup of $G$.
Let $M$ denote its Levi quotient, and $\overline{M}$ its maximal torus quotient.
Define a locally constant function $\deg_P\colon \Gr_{M,\Spd k}\to \ZZ$ as the composition
\begin{align*}
\Gr_{M,\Spd k}\longrightarrow \Gr_{\overline{M},\Spd k}
\overset{\sim}{\longrightarrow} X_*(\overline{M})
\overset{\langle 2\rho_G-2\rho_M,-\rangle}{\longrightarrow} \ZZ
\end{align*}
where $\rho_G, \rho_M$ are half sums of positive roots of $G,M$, respectively.
\begin{thm}\label{thm:CT}
Consider the shifted hyperbolic localization functor
\[
\CT_P[\deg_P]:=R(p^+_k)_!(q^+_k)^*[\deg_P]\colon \Det(\Hck_{G,\Spd k},\Lambda)^{\bd}\to \Det(\Hck_{M,\Spd k},\Lambda)^{\bd}
\]
where $p^+_k\colon \Gr_{P,\Spd k}\to \Gr_{M,\Spd k}, q^+_k\colon \Gr_{P,\Spd k}\to \Gr_{G,\Spd k}$ are natural maps.
Then $\CT_P[\deg_P]$ induces the functor
\[
\CT_P[\deg_P]\colon \Sat(\Hck_{G,\Spd k},\Lambda)\to \Sat(\Hck_{M,\Spd k},\Lambda)
\]
and there is a unique symmetric monoidal structure on $\CT_P[\deg_P]$ such that the natural isomorphism
\begin{align}\label{fml:FCT}
F_{G,\Spd k} \cong F_{M,\Spd k}\circ \CT_P[\deg_P]\colon \Sat(\Hck_{G,\Spd k},\Lambda)\to \rm{Vect}(\Lambda)
\end{align}
is monoidal.
Here the symmetric monoidal structures on the categories $\Sat(\Hck_{G,\Spd k},\Lambda)$, $\Sat(\Hck_{M,\Spd k},\Lambda)$
and the functors $F_{G,\Spd k}, F_{M,\Spd k}$ are as in Corollary \ref{cor:symmonF}.
\end{thm}
\begin{proof}
By the same argument as \cite[Proposition VI.9.6]{FS}, there is a natural symmetric monoidal structure on the functor
\[
\CT_P[\deg_P]\colon \Sat(\Hck_{G,\Spd C},\Lambda)\to \Sat(\Hck_{M,\Spd C},\Lambda)
\]
such that the natural isomorphism
\[
F_{G,\Spd C} \cong F_{M,\Spd C}\circ \CT_P[\deg_P]\colon \Sat(\Hck_{G,\Spd C},\Lambda)\to \rm{Vect}(\Lambda)
\]
is monoidal.
The following square is naturally commutative
\[
\xymatrix@C=48pt{
\Det^{\ULA}(\Hck_{G,\Spd C},\Lambda)\ar[r]^{\CT_P[\deg_P]}\ar[d]_{i^*Rj_*}& \Det^{\ULA}(\Hck_{M,\Spd C},\Lambda)\ar[d]_{i^*Rj_*}\\
\Det^{\ULA}(\Hck_{G,\Spd k},\Lambda)\ar[r]^{\CT_P[\deg_P]}& \Det^{\ULA}(\Hck_{M,\Spd k},\Lambda), 
}
\]
and the vertical arrows $i^*Rj_*$ induce equivalences of the Satake categories.
By this diagram, we can endow the functor $\CT_P[\deg_P]$ between Satake categories with a symmetric monoidal structure.
The isomorphism (\ref{fml:FCT}) is monoidal by definition.
The uniqueness of such a monoidal structure follows from the fact that $F_{G,\Spd k}$ is faithful by \cite[Definition/Proposition VI.7.10]{FS}.
\end{proof}
\section{Integral coefficient case}\label{sec:int}
In this subsection, $\Lambda$ is the ring of integers of a finite extension of $\bb{Q}_{\ell}$.
Let $\lambda\in \Lambda$ be a uniformizer.
First, recall the definition of $\Det(-,\Lambda)$.
\begin{defi}(\cite[Definition 26.1]{Sch})
For any small v-stack $Y$, define
\[
\Det(Y,\Lambda)\subset D(Y_v,\Lambda)
\]
as the full subcategory of all $A\in D(Y_v,\Lambda)$ such that $A$ is derived $(\lambda)$-complete, and $A\otimes_{\Lambda}^{\bb{L}}\Lambda/\lambda$ lies in $\Det(Y,\Lambda/\lambda)$.
\end{defi}
\begin{lemm}(\cite[Remark 26.3]{Sch})\label{lemm:modcomm}
\begin{enumerate}
\item[(i)] 
For any map $f\colon X\to Y$ of small v-stacks, the following squares are commutative:
\[
\xymatrix{
\Det(Y,\Lambda)\ar[r]^{f^*}\ar[d]_{-\otimes_{\Lambda}^{\bb{L}}\Lambda/\lambda^n}&\Det(X,\Lambda)\ar[d]_{-\otimes_{\Lambda}^{\bb{L}}\Lambda/\lambda^n}\\
\Det(Y,\Lambda/\lambda^n)\ar[r]^{f^*}&\Det(X,\Lambda/\lambda^n), 
}\quad
\xymatrix{
\Det(X,\Lambda)\ar[r]^{Rf_*}\ar[d]_{-\otimes_{\Lambda}^{\bb{L}}\Lambda/\lambda^n}&\Det(Y,\Lambda)\ar[d]_{-\otimes_{\Lambda}^{\bb{L}}\Lambda/\lambda^n}\\
\Det(X,\Lambda/\lambda^n)\ar[r]^{Rf_*}&\Det(Y,\Lambda/\lambda^n).
}
\]
\item[(ii)]
Let $f\colon X\to Y$ be a map of small v-stacks which is compactifiable, representable in locally spatial diamonds and with $\dimtrg f<\infty$.
The following squares are commutative:
\[
\xymatrix{
\Det(X,\Lambda)\ar[r]^{Rf_!}\ar[d]_{-\otimes_{\Lambda}^{\bb{L}}\Lambda/\lambda^n}&\Det(Y,\Lambda)\ar[d]_{-\otimes_{\Lambda}^{\bb{L}}\Lambda/\lambda^n}\\
\Det(X,\Lambda/\lambda^n)\ar[r]^{Rf_!}&\Det(Y,\Lambda/\lambda^n),
}\quad
\xymatrix{
\Det(Y,\Lambda)\ar[r]^{Rf^!}\ar[d]_{-\otimes_{\Lambda}^{\bb{L}}\Lambda/\lambda^n}&\Det(X,\Lambda)\ar[d]_{-\otimes_{\Lambda}^{\bb{L}}\Lambda/\lambda^n}\\
\Det(Y,\Lambda/\lambda^n)\ar[r]^{Rf^!}&\Det(X,\Lambda/\lambda^n).
}
\]
\end{enumerate}
\end{lemm}
\begin{proof}
See \cite[Remark 26.3]{Sch}.
\end{proof}
Also, the definition of universal local acyclicity is as follows:
\begin{defi}(\cite[VII.5]{FS})
Let $f\colon X\to S$ be a compactifiable map of small v-stacks representable by locally spatial diamonds with locally $\dimtrg f<\infty$.
Then we define the category of ULA complexes $\Det^{\ULA}(X/S,\Lambda)$ by $\ds\lim_{\substack{\longleftarrow\\ n}}\Det^{\ULA}(X/S,\Lambda/\lambda^n)$.
\end{defi}
\begin{defi}
Let $S$ be a small v-stack over $\Div^1$ or $\Spd \cal{O}_{\breve{F}}$.
We say that $A\in \Det(\Hck_{G,S},\Lambda)$ is ULA if its pullback to $\Gr_{G,S}$ is ULA.
Write $\Det^{\ULA}(\Hck_{G,S},\Lambda)$ for the full subcategory of $\Det(\Hck_{G,S},\Lambda)$ consisting of ULA objects.
\end{defi}
\begin{lemm}\label{lemm:ULAlim}
Let $S$ be a small v-stack over $\Div^1$ or $\Spd \cal{O}_{\breve{F}}$.
For an object $A\in \Det(\Hck_{G,S},\Lambda)$, the following conditions are equivalent:
\begin{enumerate}
\item[(i)] $A\in \Det^{\ULA}(\Hck_{G,S},\Lambda)$.
\item[(ii)] $A\otimes_{\Lambda}^{\bb{L}} \Lambda/\lambda^n \in \Det^{\ULA}(\Hck_{G,S},\Lambda/\lambda^n)$ for all $n$.
\end{enumerate}
\end{lemm}
\begin{proof}
The square
\[
\xymatrix{
\Det(\Hck_{G,S},\Lambda)\ar[d]_{-\otimes^{\bb{L}}_{\Lambda} \Lambda/\lambda^n}\ar[r]^{\pi^*}&\Det(\Gr_{G,S},\Lambda)\ar[d]^{-\otimes^{\bb{L}}_{\Lambda} \Lambda/\lambda^n}\\
\Det(\Hck_{G,S},\Lambda/\lambda^n)\ar[r]^{\pi^*}&\Det(\Gr_{G,S},\Lambda/\lambda^n)
}
\]
is commutative, where $\pi\colon \Gr_{G,S}\to \Hck_{G,S}$ is the projection.
For $A\in \Det(\Hck_{G,S},\Lambda)$, the object $\pi^*A$ is ULA if and only if $(\pi^*A)\otimes^{\bb{L}} \Lambda/\lambda^n\cong \pi^*(A\otimes^{\bb{L}} \Lambda/\lambda^n)$ is ULA for all $n$.
Hence $A$ is ULA if and only if $A\otimes^{\bb{L}} \Lambda/\lambda^n$ is ULA for all $n$.
\end{proof}
This equivalence 
implies that the essential image of $\Det^{\ULA}(\Hck_{G,S},\Lambda)$ under the categorical equivalence
\[
\Det(\Hck_{G,S},\Lambda)\cong \lim_{\substack{\longleftarrow\\ n}}\Det(\Hck_{G,S},\Lambda/\lambda^n)
\]
in \cite[Proposition 26.2]{Sch} is $\ds\lim_{\substack{\longleftarrow\\ n}} D^{\ULA/S}(\Hck_{G,S},\Lambda/\lambda^n)$.
Also, we define the flat perversity as follows.
\begin{defi}
For $A\in \Det^{\ULA}(\Hck_{G,S},\Lambda)$, 
we say that $A$ is flat perverse over $\Lambda$ if and only if $A\otimes^{\bb{L}}_{\Lambda} \Lambda/\lambda^n$ is flat perverse over $\Lambda/\lambda^n$ for all $n$.
\end{defi}
From this definition, we have
\[
\Sat(\Hck_{G,S},\Lambda)=\lim_{\substack{\longleftarrow\\ n}}\Sat(\Hck_{G,S},\Lambda/\lambda^n).
\]
By passing to the limit of the results in \S\ref{ssc:nearby}, \ref{ssc:monoidal}, \ref{ssc:TorsionRelation}, \ref{ssc:TorsionMonoidal} and using Lemma \ref{lemm:modcomm}, we get the following theorem:
\begin{thm}\label{thm:mainint}
Theorem \ref{thm:pullackequiv}, Corollary \ref{cor:preserveSat}, Proposition \ref{prop:monoidal}, Theorem \ref{thm:main} and Theorem \ref{thm:CT} hold even if $\Lambda$ is the ring of integers in a finite extension over $\QQ_{\ell}$.
\end{thm}
\section{Rational coefficient case}\label{sec:frac}
Moreover, we can prove a similar result for $\Lambda[\ell^{-1}]$-coefficient.
\begin{defi}
Let $X$ be a small v-stack.
The category $\Det(X,\Lambda[\ell^{-1}])$ is defined as the category
\[
\Det(X,\Lambda)[\ell^{-1}]
\]
obtained by inverting $\ell$ in the Hom-sets of $\Det(X,\Lambda)$.
There is a natural functor
\[
-\otimes_{\Lambda}\Lambda[\ell^{-1}]\colon \Det(X,\Lambda)\to \Det(X,\Lambda[\ell^{-1}]).
\]
\end{defi}
For any map $f\colon X\to Y$, the adjoint pair $(f^*,Rf_*)$ with coefficients in $\Lambda$ induces the adjoint pair
\begin{align*}
f^*&\colon \Det(Y,\Lambda[\ell^{-1}])\to \Det(X,\Lambda[\ell^{-1}]),\\
Rf_*&\colon \Det(X,\Lambda[\ell^{-1}])\to \Det(Y,\Lambda[\ell^{-1}]).
\end{align*}
Let $f\colon X\to Y$ be a map of small v-stacks which is compactifiable, representable in locally spatial diamonds and with $\dimtrg f<\infty$.
The adjoint pair $(Rf_!,Rf^!)$ with coefficients in $\Lambda$ induces the adjoint pair
\begin{align*}
Rf_!&\colon \Det(X,\Lambda[\ell^{-1}])\to \Det(Y,\Lambda[\ell^{-1}]),\\
Rf^!&\colon \Det(Y,\Lambda[\ell^{-1}])\to \Det(X,\Lambda[\ell^{-1}]).
\end{align*}
\begin{defi}
We say that $A\in \Det(\Hck_{G,S},\Lambda[\ell^{-1}])$ is universal locally acyclic(ULA) if $A$ is in the essential image of $\Det^{\ULA}(\Hck_{G,S},\Lambda)$ under the functor
\[
-\otimes_{\Lambda}\Lambda[\ell^{-1}]\colon \Det(\Hck_{G,S},\Lambda)\to \Det(\Hck_{G,S},\Lambda[\ell^{-1}]).
\]
Write $\Det^{\ULA}(\Hck_{G,S},\Lambda[\ell^{-1}])$ for this essential image.
\end{defi}
By definition, $\Det^{\ULA}(\Hck_{G,S},\Lambda[\ell^{-1}])$ is equivalent to the category 
\[
\Det^{\ULA}(\Hck_{G,S},\Lambda)[\ell^{-1}]
\]
Also, we define the (flat) perversity as follows.
\begin{defi}
For $A\in \Det^{\ULA}(\Hck_{G,S},\Lambda[\ell^{-1}])$, 
we say that $A$ is perverse over $\Lambda[\ell^{-1}]$ if and only if $A$ is isomorphic to an image of a flat perverse sheaf in $\Det^{\ULA}(\Hck_{G,S},\Lambda)$ under the functor $-\otimes_{\Lambda}^{\bb{L}} \Lambda[\ell^{-1}]$.
We define
\[
\Sat(\Hck_{G,S},\Lambda[\ell^{-1}])\subset \Det(\Hck_{G,S},\Lambda[\ell^{-1}])
\]
as the full subcategory of the perverse ULA sheaves.
\end{defi}
By definition, we have
\[
\Sat(\Hck_{G,S},\Lambda[\ell^{-1}])\simeq\Sat(\Hck_{G,S},\Lambda)[\ell^{-1}].
\]
By inverting $\ell$ in Theorem \ref{thm:mainint}, we get the following theorem:
\begin{thm}\label{thm:mainfrac}
Theorem \ref{thm:pullackequiv}, Corollary \ref{cor:preserveSat}, Proposition \ref{prop:monoidal}, Theorem \ref{thm:main} and Theorem \ref{thm:CT} hold even if $\Lambda$ is a finite extension over $\QQ_{\ell}$.
\end{thm}
\section{$\Qlb$-coefficient case}
We can prove the result for the $\Qlb$-coefficient case.
\begin{defi}
Let $X$ be a small v-stack.
$\Det(X,\Qlb)$ is defined by
\[
\Det(X,\Qlb)=\lim_{\substack{\longrightarrow \\ L}}\Det(X,L)
\]
where $L$ is a finite extension over $\QQ_{\ell}$, and if $L/L'/\QQ_{\ell}$ is a tower of finite extensions, the transition functor is the functor 
\[
-\otimes_{L'}L\colon \Det(X,L')\to \Det(X,L)
\]
induced by 
\[
-\otimes_{\cal{O}_{L'}}\cal{O}_L\colon \Det(X,\cal{O}_{L'})\to \Det(X,\cal{O}_L).
\]
There is a natural functor
\[
-\otimes_L \Qlb\colon \Det(X,L)\to \Det(X,\Qlb).
\]
\end{defi}
For any map $X\to Y$ of small v-stacks, the adjoint pairs $(f^*,Rf_*)$ with coefficient in a finite extension $L$ over $\QQ_{\ell}$ induce the adjoint pair
\begin{align*}
f^*&\colon \Det(Y,\Qlb)\to \Det(X,\Qlb),\\
Rf_*&\colon \Det(X,\Qlb)\to \Det(Y,\Qlb).
\end{align*}
In fact, the following lemma holds:
\begin{lemm}\label{lemm:LLcompati}
Let $L/L'/\QQ_{\ell}$ be a tower of finite extensions.
\begin{enumerate}
\item[(i)]
For any map $f\colon X\to Y$ of small v-stacks, the following squares are commutative:
\[
\xymatrix{
\Det(Y,L')\ar[r]^{f^*}\ar[d]_{-\otimes_{L'}L}&\Det(X,L')\ar[d]_{-\otimes_{L'}L}\\
\Det(Y,L)\ar[r]^{f^*}&\Det(X,L),
}\quad
\xymatrix{
\Det(X,L')\ar[r]^{Rf_*}\ar[d]_{-\otimes_{L'}L}&\Det(Y,L')\ar[d]_{-\otimes_{L'}L}\\
\Det(X,L)\ar[r]^{Rf_*}&\Det(Y,L).
}
\]
\item[(ii)]
Let $f\colon X\to Y$ be a map of small v-stacks which is compactifiable, representable in locally spatial diamonds and with $\dimtrg f<\infty$.
The following squares are commutative:
\[
\xymatrix{
\Det(X,L')\ar[r]^{Rf_!}\ar[d]_{-\otimes_{L'}L}&\Det(Y,L')\ar[d]_{-\otimes_{L'}L}\\
\Det(X,L)\ar[r]^{Rf_!}&\Det(Y,L),
}\quad
\xymatrix{
\Det(Y,L')\ar[r]^{Rf^!}\ar[d]_{-\otimes_{L'}L}&\Det(X,L')\ar[d]_{-\otimes_{L'}L}\\
\Det(Y,L)\ar[r]^{Rf^!}&\Det(X,L).
}
\]
\end{enumerate}
\end{lemm}
\begin{proof}
For (i), we need to show that the squares
\[
\xymatrix{
\Det(Y,\cal{O}_{L'})\ar[r]^{f^*}\ar[d]_{-\otimes_{\cal{O}_{L'}}\cal{O}_{L}}&\Det(X,\cal{O}_{L'})\ar[d]_{-\otimes_{\cal{O}_{L'}}\cal{O}_{L}}\\
\Det(Y,\cal{O}_{L})\ar[r]^{f^*}&\Det(X,\cal{O}_{L}),
}
\xymatrix{
\Det(X,\cal{O}_{L'})\ar[r]^{Rf_*}\ar[d]_{-\otimes_{\cal{O}_{L'}}\cal{O}_{L}}&\Det(Y,\cal{O}_{L'})\ar[d]_{-\otimes_{\cal{O}_{L'}}\cal{O}_{L}}\\
\Det(X,\cal{O}_{L})\ar[r]^{Rf_*}&\Det(Y,\cal{O}_{L})
}
\]
are naturally commutative.
By Lemma \ref{lemm:modcomm}, it is enough to prove that the squares
\[
\xymatrix{
\Det(Y,\cal{O}_{L'}/\ell^n)\ar[r]^{f^*}\ar[d]_{-\otimes_{\cal{O}_{L'}/\ell^n}\cal{O}_{L}/\ell^n}&\Det(X,\cal{O}_{L'}/\ell^n)\ar[d]_{-\otimes_{\cal{O}_{L'}/\ell^n}\cal{O}_{L}/\ell^n}\\
\Det(Y,\cal{O}_{L}/\ell^n)\ar[r]^{f^*}&\Det(X,\cal{O}_{L}/\ell^n),
}
\xymatrix{
\Det(X,\cal{O}_{L'}/\ell^n)\ar[r]^{Rf_*}\ar[d]_{-\otimes_{\cal{O}_{L'}/\ell^n}\cal{O}_{L}/\ell^n}&\Det(Y,\cal{O}_{L'}/\ell^n)\ar[d]_{-\otimes_{\cal{O}_{L'}/\ell^n}\cal{O}_{L}/\ell^n}\\
\Det(X,\cal{O}_{L}/\ell^n)\ar[r]^{Rf_*}&\Det(Y,\cal{O}_{L}/\ell^n)
}
\]
are naturally commutative.
The commutativity of the first square follows from the definition.
For the second square, since $\cal{O}_{L}/\ell^n$ is a finite free module over $\cal{O}_{L'}/\ell^n$, the outer square in the diagram
\[
\xymatrix{
\Det(X,\cal{O}_{L'}/\ell^n)\ar[r]^{Rf_*}\ar[d]_{-\otimes_{\cal{O}_{L'}/\ell^n}\cal{O}_{L}/\ell^n}&\Det(Y,\cal{O}_{L'}/\ell^n)\ar[d]_{-\otimes_{\cal{O}_{L'}/\ell^n}\cal{O}_{L}/\ell^n}\\
\Det(X,\cal{O}_{L}/\ell^n)\ar[r]^{Rf_*}\ar[d]_{\rm{For}}&\Det(Y,\cal{O}_{L}/\ell^n)\ar[d]^{\rm{For}}\\
\Det(X,\cal{O}_{L'}/\ell^n)\ar[r]^{Rf_*}&\Det(Y,\cal{O}_{L'}/\ell^n)
}
\]
is commutative, where the vertical functors $\rm{For}$ are forgetful functors.
As the functor
\[
\rm{For}\colon \Det(Y,\cal{O}_{L}/\ell^n)\to \Det(Y,\cal{O}_{L'}/\ell^n)
\]
is conservative, it suffices to show that the following square is naturally commutative:
\[
\xymatrix{
\Det(X,\cal{O}_{L}/\ell^n)\ar[r]^{Rf_*}\ar[d]_{\rm{For}}&\Det(Y,\cal{O}_{L}/\ell^n)\ar[d]_{\rm{For}}\\
\Det(X,\cal{O}_{L'}/\ell^n)\ar[r]^{Rf_*}&\Det(Y,\cal{O}_{L'}/\ell^n).
}
\]
This square is the right adjoint of the first square.
The proof of (ii) is similar.
\end{proof}
\begin{defi}
We say that $A\in \Det(\Hck_{G,S},\Qlb)$ is universal locally acyclic (ULA) if there exists a finite extension $L$ over $\bb{Q}_{\ell}$ such that $A$ is in the essential image of $\Det^{\ULA}(\Hck_{G,S},L)$ under the functor
\[
-\otimes_L \Qlb\colon \Det(\Hck_{G,S},L)\to \Det(\Hck_{G,S},\Qlb).
\]
Write $\Det^{\ULA}(\Hck_{G,S},\Qlb)$ for the full subcategory of ULA complexes in $\Det(\Hck_{G,S},\Qlb)$.
\end{defi}
By definition, it holds that
\[
\Det^{\ULA}(\Hck_{G,S},\Qlb)=\lim_{\substack{\longrightarrow \\ L}}\Det^{\ULA}(\Hck_{G,S},L).
\]
We define the perversity as follows.
\begin{defi}
We say that $A\in \Det^{\ULA}(\Hck_{G,S},\Qlb)$ is perverse if there exists a finite extension $L$ over $\bb{Q}_{\ell}$ such that $A$ is in the essential image of $\Sat(\Hck_{G,S},L)$ under the functor
\[
-\otimes_{L}\Qlb\colon \Det^{\ULA}(\Hck_{G,S},L)\to \Det^{\ULA}(\Hck_{G,S},\Qlb).
\]
We define
\[
\Sat(\Hck_{G,S},\Lambda[\ell^{-1}])\subset \Det(\Hck_{G,S},\Lambda[\ell^{-1}])
\]
as the full subcategory of the perverse ULA sheaves.
\end{defi}
By definition, it holds that
\[
\Sat(\Hck_{G,S},\Qlb)=\lim_{\substack{\longrightarrow \\ L}}\Sat(\Hck_{G,S},L).
\]
By passing the limit in Theorem \ref{thm:mainfrac} and using Lemma \ref{lemm:LLcompati}, we get the following theorem:
\begin{thm}\label{thm:mainQlb}
Theorem \ref{thm:pullackequiv}, Corollary \ref{cor:preserveSat}, Proposition \ref{prop:monoidal}, Theorem \ref{thm:main} and Theorem \ref{thm:CT} hold even if $\Lambda=\Qlb$.
\end{thm}
\bibliographystyle{test}
\bibliography{twogeomsatbib}

\end{document}